\documentclass[leqno,12pt]{amsart}
\usepackage{amsfonts}
\usepackage{amsmath,amssymb,amsthm}
\newcommand{\R}{\mathbb R}
\newcommand{\E}{\mathbb E}

\newcommand{\tr}{\mathrm{tr}}

\newtheorem{thm}{Theorem}[section]

\newtheorem{prop}[thm]{Proposition}
\theoremstyle{definition}
\newtheorem{defn}[thm]{Definition}
\theoremstyle{remark}

\newcommand{\ds}{\displaystyle}

\begin{document}

\title[Lorentz Surfaces with Parallel Normalized Mean Curvature Vector]
{On the Theory of Lorentz Surfaces with Parallel Normalized Mean Curvature Vector Field
in Pseudo-Euclidean 4-Space}

\author{Yana Aleksieva, Georgi Ganchev, Velichka Milousheva}

\address{Faculty of Mathematics and Informatics, Sofia University,
5 James Bourchier blvd., 1164 Sofia, Bulgaria}
\email{yana\_a\_n@fmi.uni-sofia.bg}
\address{Institute of Mathematics and Informatics, Bulgarian Academy of Sciences,
Acad. G. Bonchev Str. bl. 8, 1113 Sofia, Bulgaria}
\email{ganchev@math.bas.bg}
\address{Institute of Mathematics and Informatics, Bulgarian Academy of Sciences,
Acad. G. Bonchev Str. bl. 8, 1113, Sofia, Bulgaria;
"L. Karavelov" Civil Engineering Higher School, 175 Suhodolska Str., 1373 Sofia, Bulgaria}
\email{vmil@math.bas.bg}


\begin{abstract}
We develop an invariant local theory of Lorentz surfaces in pseudo-Euclidean 4-space by use of a linear map
of Weingarten type. We find a geometrically determined moving frame field at each point of the surface and
obtain a system of geometric functions. We  prove a fundamental existence and uniqueness theorem in terms of
these functions. On any Lorentz surface with parallel normalized mean curvature vector field we introduce
special geometric (canonical) parameters and prove that any such  surface is determined  up to a rigid
motion by three invariant functions satisfying three natural partial differential equations.
In this way we minimize the number of functions and the number of partial differential equations determining
the surface, which solves the Lund-Regge problem for this class of surfaces.
\end{abstract}

\maketitle

\section{Introduction}

In pseudo-Euclidean spaces there are two types of surfaces according to their induced metric - Riemannian
or Lorentz metric. In the present paper we study Lorentz surfaces in pseudo-Euclidean space $\E^4_2$.

Recently, many classification results for Lorentz surfaces in pseudo-Euclidean spaces have been obtained
imposing some extra conditions on the mean curvature vector, the Gauss curvature, or the second fundamental form.

In \cite{Chen3}  B.-Y. Chen obtained several classification results for minimal Lorentz surfaces
in indefinite space forms. In particular, he completely classified all minimal Lorentz surfaces
in a pseudo-Euclidean space $\E^m_s$ with arbitrary dimension $m$ and arbitrary index $s$.

A natural extension of minimal surfaces are quasi-minimal (or marginally trapped) surfaces - these
are surfaces whose mean curvature vector is lightlike at each point of the surface.
Quasi-minimal surfaces in pseudo-Euclidean space have been very actively studied in the last few years.
In \cite{Chen-JMAA}  B.-Y. Chen classified
quasi-minimal Lorentz  flat surfaces in $\E^4_2$ and gave a complete classification of biharmonic Lorentz surfaces
in  $\E^4_2$ with lightlike mean curvature vector. Several other
families of quasi-minimal surfaces  have also been classified. For
example, quasi-minimal surfaces with constant Gauss curvature in
$\E^4_2$ were classified in \cite{Chen-HMJ, Chen-Yang}. Quasi-minimal
Lagrangian surfaces and quasi-minimal slant surfaces in complex
space forms were classified, respectively, in \cite{Chen-Dillen} and
\cite{Chen-Mihai}. The classification of quasi-minimal surfaces with parallel mean
curvature vector in  $\E^4_2$ is obtained
in \cite{Chen-Garay}.
In \cite{GM5} the classification of quasi-minimal rotational
surfaces of elliptic, hyperbolic or parabolic type is given.
For an up-to-date survey on  quasi-minimal
surfaces, see also \cite{Chen-TJM}.

A Lorentz surface of an indefinite space form is called parallel if its second fundamental form is parallel
with respect to the Van der Waerden-Bortolotti connection.  Parallel surfaces are important in differential
geometry as well as in physics since extrinsic invariants of such surfaces do not change from point to point.
Parallel Lorentz surfaces in four-dimensional Lorentzian space forms were studied by B.-Y. Chen and
J. Van der Veken in \cite{Chen-Veken}. An explicit classification of parallel Lorentz surfaces in
the pseudo-Euclidean space $\E^4_2$, in the pseudo-hyperbolic space $\mathbb{H}^4_2(-1)$, and in the neutral
pseudo-sphere $\mathbb{S}^4_2(1)$ is given by Chen et al. in \cite{Chen-Dillen-Veken}, \cite{Chen-CEJM},
and \cite{Chen-JMP}, respectively. The complete classification of parallel Lorentz surfaces in a
pseudo-Euclidean space with arbitrary codimension and arbitrary index is obtained in \cite{Chen-JGP}.

Another basic class of surfaces in Riemannian and pseudo-Riemannian geometry are the surfaces with parallel
mean curvature vector field, since they are critical points of some natural functionals and  play important
role in differential geometry,  the theory of harmonic maps, as well as in physics.
Surfaces with parallel mean curvature vector field in Riemannian space forms were classified in the early
1970s by Chen \cite{Chen1} and Yau  \cite{Yau}. Recently, spacelike surfaces with parallel mean
curvature vector field  in arbitrary indefinite space forms were classified in \cite{Chen1-2} and  \cite{Chen1-3}.
A complete classification of Lorentz surfaces with parallel mean curvature vector field in arbitrary
pseudo-Euclidean space $\E^m_s$ is given in \cite{Chen-KJM} and \cite{Fu-Hou}. A survey on classical and
recent results concerning submanifolds with parallel mean curvature vector in Riemannian manifolds
as well as in pseudo-Riemannian manifolds is presented in \cite{Chen-survey}.

A natural extension of the class of surfaces with parallel mean curvature vector field are surfaces with parallel
normalized mean curvature vector field.
A surface $M$ in a Riemannian manifold is said to have parallel
normalized mean curvature vector field  if the mean curvature vector $H$ is non-zero and the
unit vector in the direction of the mean curvature vector is parallel in the normal
bundle  \cite{Chen-MM}.
The condition to have parallel normalized mean curvature vector field is  weaker than the condition
to have parallel mean curvature vector field.
It is known that every surface in the Euclidean 3-space has parallel normalized mean
curvature vector field but in the 4-dimensional Euclidean space, there exist abundant examples of surfaces
which lie fully in $\E^4$ with parallel normalized mean
curvature vector field, but not with parallel mean curvature vector field.

In \cite{Chen-MM} it is proved that every analytic surface  with parallel normalized mean curvature vector
in the Euclidean space $\E^m$ must either lie in a 4-dimensional space $\E^4$ or in a hypersphere of $\E^m$
as a minimal surface. Spacelike submanifolds with parallel normalized mean curvature vector field in a
general de Sitter space are studied in \cite{Shu}. It is shown that compact spacelike submanifolds whose
mean curvature does not vanish and whose corresponding normalized vector field is parallel, must be, under
some suitable geometric assumptions, totally umbilical.

In the present paper we study the local theory of Lorentz surfaces with parallel normalized mean curvature
vector field in the pseudo-Euclidean space $\E^4_2$.
Our approach to the study of these surfaces is based on the introduction of canonical parameters.

In Section \ref{S:Weingarten} we develop an invariant theory of Lorentz surfaces in $\E^4_2$ similarly
to the theory of surfaces in the Euclidean space $\E^4$ and the theory of spacelike surfaces in
the Minkowski space $\E^4_1$.
We introduce an invariant linear map $\gamma$ of Weingarten-type  in the tangent plane at
any point of the surface, which generates two invariant functions $k = \det \gamma$ and
$\varkappa= -\ds{ \frac{1}{2}}\, \tr \gamma$.
In the case $\varkappa^2 - k > 0$ we introduce principal lines and a
geometrically determined moving frame field at each point of the
surface. Writing derivative formulas of Frenet-type for this frame
field, we obtain a system of geometric  functions  and prove a fundamental existence and
uniqueness theorem, stating that these functions  determine the surface up to a rigid motion  in $\E^4_2$.

The basic geometric classes of surfaces in $\E^4_2$ such as quasi-minimal surfaces, surfaces with
flat normal connection, surfaces with constant Gauss curvature or constant normal curvature, surfaces
with parallel mean curvature vector field, etc., are characterized by conditions on their geometric functions.

We focus our attention on the class of surfaces with parallel normalized mean curvature vector field.
We introduce canonical parameters on each such surface that allow us to formulate the fundamental
existence and uniqueness theorem in terms of three invariant functions. Our main result states that
any Lorentz surface with parallel normalized mean curvature vector field is determined up to a rigid
motion in $\E^4_2$   by three invariant functions satisfying a system of three natural partial differential
equations (Theorem \ref{T:Fundamental Theorem-canonical}). This theorem solves the Lund-Regge problem for
the class of surfaces with parallel normalized mean curvature vector field in $\E^4_2$.

\section{Preliminaries}

We consider the  pseudo-Euclidean 4-dimensional space $\E^4_2$ endowed with the canonical pseudo-Euclidean
metric of index 2 given in local coordinates by $$g_0 = dx_1^2 + dx_2^2 - dx_3^2 - dx_4^2,$$
where $(x_1, x_2, x_3, x_4)$ is a rectangular coordinate system of $\E^4_2$. As usual, we denote by
$\langle \, , \rangle$ the indefinite inner scalar product with respect to $g_0$.

 A vector $v$ is said to be  \emph{spacelike} (respectively, \emph{timelike}) if $\langle v, v \rangle > 0$
 (respectively, $\langle v, v \rangle < 0$).
 A vector $v$ is called \emph{lightlike} if it is nonzero and satisfies $\langle v, v \rangle = 0$.

A surface $M^2_1$ in $\E^4_2$ is called \emph{Lorentz}  if the
induced  metric $g$ on $M^2_1$ is Lorentzian. So, at each point $p\in M^2_1$ we have the following decomposition
$$\E^4_2 = T_pM^2_1 \oplus N_pM^2_1$$
with the property that the restriction of the metric
onto the tangent space $T_pM^2_1$ is of
signature $(1,1)$, and the restriction of the metric onto the normal space $N_pM^2_1$ is of signature $(1,1)$.

We denote by $\nabla$ and $\nabla'$ the Levi Civita connections of $M^2_1$  and $\E^4_2$, respectively.
For vector fields $x$, $y$  tangent to $M^2_1$ and a vector field $\xi$ normal to $M^2_1$, the formulas
of Gauss and Weingarten, giving a decomposition of the vector fields $\nabla'_xy$ and
$\nabla'_x \xi$ into tangent and  normal components, are given respectively by \cite{Chen1}:
$$\begin{array}{l}
\vspace{2mm}
\nabla'_xy = \nabla_xy + \sigma(x,y);\\
\vspace{2mm}
\nabla'_x \xi = - A_{\xi} x + D_x \xi.
\end{array}$$
These formulas define the second fundamental form $\sigma$, the normal
connection $D$, and the shape operator $A_{\xi}$ with respect to
$\xi$.
For each normal vector field $\xi$, the shape operator $A_{\xi}$ is a symmetric endomorphism of the
tangent space $T_pM^2_1$ at $p \in M^2_1$.
  In general, $A_{\xi}$ is not diagonalizable.
It is well known that the shape operator and the second fundamental form are related by the formula
$$\langle \sigma(x,y), \xi \rangle = \langle A_{\xi} x, y \rangle$$
for $x$, $y$ tangent to $M^2_1$ and $\xi$ normal to $M^2_1$.

The mean curvature vector  field $H$ of $M^2_1$ in $\E^4_2$
is defined as $H = \ds{\frac{1}{2}\,  \tr\, \sigma}$.
The  surface $M^2_1$  is called \emph{minimal} if its mean curvature vector vanishes identically, i.e. $H =0$.
A natural extension of minimal surfaces are quasi-minimal surfaces.
The surface $M^2_1$  is called \emph{quasi-minimal} (or \textit{pseudo-minimal}) if its
mean curvature vector is lightlike at each point, i.e. $H \neq 0$ and $\langle H, H \rangle =0$  \cite{Rosca}.
Obviously, quasi-minimal surfaces are always non-minimal.

A normal vector field $\xi$ on $M^2_1$ is called \emph{parallel in the normal bundle}
(or simply \emph{parallel}) if $D{\xi}=0$ holds identically \cite{Chen2}.
The surface $M^2_1$ is said to have \emph{parallel mean curvature vector field} if its mean curvature vector $H$
satisfies $D H =0$ identically.

Surfaces for which the mean curvature vector $H$ is non-zero and  there exists a parallel unit vector field
$b$ in the direction of the mean curvature vector
 $H$, such that $b$ is parallel in the normal
bundle, are called surfaces with \textit{parallel normalized mean curvature vector field} \cite{Chen-MM}.
It is easy to see  that if $M^2_1$ is a surface  with non-zero parallel mean curvature vector field $H$
(i.e. $DH = 0$), then $M^2_1$ is a surface with parallel normalized mean curvature vector field,
but the converse is not true in general.
It is true only in the case $\Vert H \Vert = const$.

A submanifold  $M^2_1$  of a pseudo-Riemannian manifold is called \emph{totally geodesic} if the second
fundamental form $\sigma$ of  $M^2_1$
vanishes identically. It is called \emph{totally umbilical} if its second fundamental form satisfies
$\sigma(x,y) = \langle x, y \rangle \,H$
for arbitrary vector fields $x$, $y$  tangent to $M^2_1$.

\section{Weingarten map of a Lorentz surface in $\E^4_2$} \label{S:Weingarten}

Let $M^2_1: z=z(u,v), \,\, (u,v) \in \mathcal{D}$ $(\mathcal{D} \subset \R^2)$  be a local parametrization
on a Lorentz surface in $\E_{2}^{4}$.
The tangent
space $T_pM^{2}$ at an  arbitrary point $p=z(u,v)$ of $M^{2}$ is spanned
by the vector fields $z_u$ and $z_v$. We assume that $\langle z_u,z_u\rangle >0$
and $\langle z_v,z_v\rangle <0$. Hence the coefficients
of the first fundamental form of $M^2_1$ are
$E=\langle z_u,z_u\rangle$, $F=\langle z_u,z_v \rangle$,
$G=\langle z_v,z_v\rangle$, where $E>0$ and $G<0$. We denote $W = \sqrt{|EG - F^2|}$.

We consider a normal frame field $\{n_{1},n_{2}\}$ satisfying
the conditions $\left\langle n_{1},n_{1}\right\rangle =1$, $\left\langle n_{1},n_{2}\right\rangle =0$,
$\left\langle n_{2},n_{2}\right\rangle =-1$.
Then we
have the following derivative formulas:
\begin{equation} \label{E:eq0}
\begin{array}{l}
\vspace{2mm}\nabla_{z_{u}}^{'}z_{u}=z_{uu}=\Gamma_{11}^{1}z_{u}-\Gamma_{11}^{2}z_{v}+c_{11}^{1}n_{1}-c_{11}^{2}n_{2}\\
\vspace{2mm}\nabla_{z_{u}}^{'}z_{v}=z_{uv}=\Gamma_{12}^{1}z_{u}-\Gamma_{12}^{2}z_{v}+c_{12}^{1}n_{1}-c_{12}^{2}n_{2}\\
\vspace{2mm}\nabla_{z_{v}}^{'}z_{v}=z_{vv}=\Gamma_{22}^{1}z_{u}-\Gamma_{22}^{2}z_{v}+c_{22}^{1}n_{1}-c_{22}^{2}n_{2},
\end{array}
\end{equation}
where $\Gamma_{ij}^{k}$ are the Christoffel's symbols and the functions $c_{ij}^k, \,\, i,j,k = 1,2$ are given by

$$\begin{array}{ll}
\vspace{2mm}
c_{11}^{1}= \langle z{}_{uu},n{}_{1} \rangle; & c_{11}^{2}= \langle z{}_{uu},n{}_{2} \rangle; \\
\vspace{2mm}
c_{12}^{1}= \langle z{}_{uv},n{}_{1} \rangle;& c_{12}^{2}= \langle z{}_{uv},n{}_{2} \rangle;\\
\vspace{2mm}
c_{22}^{1}= \langle z{}_{vv},n{}_{1} \rangle;& c_{22}^{2}= \langle z{}_{vv},n{}_{2} \rangle.
\end{array}$$
\vskip 2mm

Obviously, if  $c_{ij}^k=0, \; i,j,k = 1, 2$ then
$M^2_1$ is totally geodesic and hence,  the surface $M^2_1$ is an open part of a pseudo-Euclidean
linear subspace of $\E^4_2$, i.e. $M^2_1$ is part of a 2-plane. So,
we assume that at least one of the coefficients $c_{ij}^k$ is not
zero.

It follows from \eqref{E:eq0} that
\begin{equation}\label{E:eq1}
\begin{array}{l}
\vspace{2mm}
\sigma\left(z_{u},z_{u}\right)=c_{11}^{1}n_{1}-c_{11}^{2}n_{2};\\
\vspace{2mm}
\sigma\left(z_{u},z_{v}\right)=c_{12}^{1}n_{1}-c_{12}^{2}n_{2};\\
\vspace{2mm}
\sigma\left(z_{v},z_{v}\right)=c_{22}^{1}n_{1}-c_{22}^{2}n_{2}.
\end{array}
\end{equation}

\vskip 2mm
Using $c_{ij}^{k}$ we introduce the following functions:

$$\Delta_1 = \left|%
\begin{array}{cc}
\vspace{2mm}
  c_{11}^1 & c_{12}^1 \\
  c_{11}^2 & c_{12}^2 \\
\end{array}%
\right|; \quad
\Delta_2 = \left|%
\begin{array}{cc}
\vspace{2mm}
  c_{11}^1 & c_{22}^1 \\
  c_{11}^2 & c_{22}^2 \\
\end{array}%
\right|; \quad
\Delta_3 = \left|%
\begin{array}{cc}
\vspace{2mm}
  c_{12}^1 & c_{22}^1 \\
  c_{12}^2 & c_{22}^2 \\
\end{array}%
\right|;$$

$$ \ds{L(u,v)=\frac{2\Delta_{1}}{W}}; \qquad \ds{M(u,v)=\frac{\Delta_{2}}{W}}; \qquad
\ds{N(u,v)=\frac{2\Delta_{3}}{W}}.$$
\vskip 2mm
Let
$$\begin{array}{l}
\vspace{2mm}u=u(\bar{u},\bar{v});\\
\vspace{2mm}v=v(\bar{u},\bar{v}),
\end{array}\quad(\bar{u},\bar{v})\in\bar{\mathcal{D}},\,\,\bar{\mathcal{D}}\subset\R^{2},$$
be a smooth change of  the parameters $(u, v)$ on $M^{2}$
with $J=u_{\bar{u}}\,v_{\bar{v}}-u_{\bar{v}}\,v_{\bar{u}}\neq0$. Then
\vskip 2mm

\begin{equation}\label{E:eq2}
\begin{array}{l}
\vspace{2mm}z_{\bar{u}}=z_{u}\,u_{\bar{u}}+z_{v}\,v_{\bar{u}},\\
\vspace{2mm}z_{\bar{v}}=z_{u}\,u_{\bar{v}}+z_{v}\,v_{\bar{v}}.
\end{array}
\end{equation}

If $\overline{E}=\langle z_{\bar{u}},z_{\bar{u}}\rangle$, $\overline{F}=\langle z_{\bar{u}},z_{\bar{v}}\rangle$
and $\overline{G}=\langle z_{\bar{v}},z_{\bar{v}}\rangle$, then we get
$$\begin{array}{l}
\vspace{2mm}\overline{E}=u_{\bar{u}}^{2}\,E+2\,u_{\bar{u}}v_{\bar{u}}\,F+v_{\bar{u}}^{2}\,G,\\
\vspace{2mm}\overline{F}=u_{\bar{u}}u_{\bar{v}}\,E+(u_{\bar{u}}v_{\bar{v}}+v_{\bar{u}}u_{\bar{v}})\,F+v_{\bar{u}}v_{\bar{v}}\,G,\\
\vspace{2mm}\overline{G}=u_{\bar{v}}^{2}\,E+2\,u_{\bar{v}}v_{\bar{v}}\,F+v_{\bar{v}}^{2}\,G
\end{array}$$
and  hence, $\overline{E}\,\overline{G}-\overline{F}^{2}=J^{2}\,(EG-F^{2})$,
$\overline{W}=\varepsilon J\,W$, where $\varepsilon={\rm sign}\,J$.

Further we calculate the functions $\bar{c}_{ij}^{k}$ which take part in the following equalities:
\begin{equation}\label{E:eq2-a}
\begin{array}{l}
\vspace{2mm}\sigma(z_{\bar{u}},z_{\bar{u}})=\bar{c}_{11}^{1}\,n_{1}-\bar{c}_{11}^{2}\,n_{2};\\
\vspace{2mm}\sigma(z_{\bar{u}},z_{\bar{v}})=\bar{c}_{12}^{1}\,n_{1}-\bar{c}_{12}^{2}\,n_{2};\\
\vspace{2mm}\sigma(z_{\bar{v}},z_{\bar{v}})=\bar{c}_{22}^{1}\,n_{1}-\bar{c}_{22}^{2}\,n_{2}.
\end{array}
\end{equation}

Using  \eqref{E:eq1}, \eqref{E:eq2}, and \eqref{E:eq2-a} we find
$$\begin{array}{l}
\vspace{2mm}\bar{c}_{11}^{k}=u_{\bar{u}}^{2}\,c_{11}^{k}+
2u_{\bar{u}}\,v_{\bar{u}}\,c_{12}^{k}+v_{\bar{u}}^{2}\,c_{22}^{k},\\
\vspace{2mm}\bar{c}_{12}^{k}=u_{\bar{u}}\,u_{\bar{v}}\,c_{11}^{k}+
(u_{\bar{u}}\,v_{\bar{v}}+u_{\bar{v}}\,v_{\bar{u}})\,c_{12}^{k}+v_{\bar{u}}\,v_{\bar{v}}\,c_{22}^{k},\\
\vspace{2mm}\bar{c}_{22}^{k}=u_{\bar{v}}^{2}\,c_{11}^{k}+2u_{\bar{v}}\,v_{\bar{v}}\,c_{12}^{k}+
v_{\bar{v}}^{2}\,c_{22}^{k}.
\end{array}\quad\quad(k=1,2),$$
and hence
$$\begin{array}{l}
\vspace{2mm}\bar{\Delta}_{1}=J\left(u_{\bar{u}}^{2}\,\Delta_{1}+u_{\bar{u}}\,v_{\bar{u}}\,\Delta_{2}+
v_{\bar{u}}^{2}\,\Delta_{3}\right);\\
\vspace{2mm}\bar{\Delta}_{2}=J\left(2u_{\bar{u}}\,u_{\bar{v}}\,\Delta_{1}+(u_{\bar{u}}\,v_{\bar{v}}+
u_{\bar{v}}\,v_{\bar{u}})\,\Delta_{2}+2v_{\bar{u}}\,v_{\bar{v}}\,\Delta_{3}\right);\\
\vspace{2mm}\bar{\Delta}_{3}=J\left(u_{\bar{v}}^{2}\,\Delta_{1}+u_{\bar{v}}\,v_{\bar{v}}\,\Delta_{2}+
v_{\bar{v}}^{2}\,\Delta_{3}\right).
\end{array}$$

Thus we find that the functions $\overline{L}$, $\overline{M}$, $\overline{N}$
are expressed as follows:
$$\begin{array}{l}
\vspace{2mm}\overline{L}=\varepsilon(u_{\bar{u}}^{2}\,L+2\,u_{\bar{u}}v_{\bar{u}}\,M+v_{\bar{u}}^{2}\,N),\\
\vspace{2mm}\overline{M}=\varepsilon(u_{\bar{u}}u_{\bar{v}}\,L+(u_{\bar{u}}v_{\bar{v}}+
v_{\bar{u}}u_{\bar{v}})\,M+v_{\bar{u}}v_{\bar{v}}\,N),\\
\vspace{2mm}\overline{N}=\varepsilon(u_{\bar{v}}^{2}\,L+2\,u_{\bar{v}}v_{\bar{v}}\,M+v_{\bar{v}}^{2}\,N).
\end{array}$$

Hence, the functions $L,M,N$ change in the same way as the coefficients
of the first fundamental form $E,F,G$ under any change of the parameters
on $M^2_1$.

If we take a vector field $X\in T_pM^2_1$, such that $X=\lambda z_{u}+\mu z_{v}=
\bar{\lambda}z_{\bar{u}}+\bar{\mu}z_{\bar{v}}$,
then $\lambda=u_{\bar{u}}\bar{\lambda}+u_{\bar{v}}\bar{\mu},\,\,\mu=
v_{\bar{u}}\bar{\lambda}+v_{\bar{v}}\bar{\mu}$.
So, for each  two tangent vector fields $X_1 =\lambda_1 z_{u}+\mu_1 z_{v}$ and $X_2=\lambda_2 z_{u}+\mu_2 z_{v}$
we can consider the quadratic form with coefficients $L, M, N$ and we have the following equality:

$$\ds{\overline{L}\bar{\lambda}_{1}\bar{\lambda}_{2}+\overline{M}(\bar{\lambda}_{1}\bar{\mu}_{2}+
\bar{\mu}_{1}\bar{\lambda}_{2})+\overline{N}\bar{\mu}_{1}\bar{\mu}_{2}=
\varepsilon\left(L\lambda_{1}\lambda_{2}+M(\lambda_{1}\mu_{2}+\mu_{1}\lambda_{2})+N\mu_{1}\mu_{2}\right)}.$$
\vskip 2mm

The last formula allows us to define second fundamental form
$II$ of the surface $M^2_1$ at $p\in M^2_1$ as follows. Let $X=\lambda z_{u}+\mu z_{v},\,\,(\lambda,\mu)\neq(0,0)$
be a tangent vector at a point $p\in M^2_1$. Then

$$II(\lambda,\mu)=L\lambda^{2}+2M\lambda\mu+N\mu^{2},\quad\lambda,\mu\in{\R}.$$

\vskip 2mm
Further we will show that the functions $L, M, N$ do not depend on the choice of the normal frame
of the surface.  Let
$\{\widetilde{n}_1, \widetilde{n}_2\}$ be another normal frame
field of $M^2_1$, such that
$\langle \widetilde{n}_1, \widetilde{n}_1 \rangle = 1$, $\langle \widetilde{n}_1,
\widetilde{n}_2 \rangle = 0$, $\langle \widetilde{n}_2, \widetilde{n}_2 \rangle = -1$. The relation between
the two normal frame fields is given by
$$\begin{array}{l}
\vspace{2mm}n_{1}=\varepsilon'(\cosh\theta\,\widetilde{n}_{1}+\sinh\theta\,\widetilde{n}_{2});\\
\vspace{2mm}n_{2}=\varepsilon'(\sinh\theta\,\widetilde{n}_{1}+\cosh\theta\,\widetilde{n}_{2}),
\end{array}\qquad\varepsilon'=\pm1$$
for some smooth function $\theta$, and the relation between the corresponding functions $c_{ij}^{k}$ and
$\widetilde{c}_{ij}^{k}$, $i,j,k=1,2$ is
$$\begin{array}{l}
\vspace{2mm}\widetilde{c}_{ij}^{1}=\varepsilon'(\cosh\theta\,c_{ij}^{1}-\sinh\theta\,c_{ij}^{2});\\
\vspace{2mm}\widetilde{c}_{ij}^{2}=\varepsilon'(-\sinh\theta\,c_{ij}^{1}+\cosh\theta\,c_{ij}^{2}),
\end{array}\quad i,j=1,2.$$
Thus, $\widetilde{\Delta}_{i}=\Delta_{i}$, $i=1,2,3$, and $\widetilde{L}=L,\;\widetilde{M}=M,\;\widetilde{N}=N$.
So, the functions $L$, $M$, $N$ do not depend on the normal frame
of the surface.

Hence, the second fundamental form $II$ is invariant
up to the orientation of the tangent space or the normal space of
the surface.

\vskip 1mm
Such a bilinear  form has been considered for an arbitrary
2-dimensional surface in a 4-dimensional affine space
$\mathbb{A}^4$ (see for example \cite{Bur-Mayer,Lane,Walter}). Here we use
this form for Lorentz surfaces
in $\E^4_2$ and taking into consideration also the first fundamental form
we develop the theory of Lorentz surfaces similar to the theory of surfaces in $\E^4$ and $\E^4_1$.

It follows from  formulas \eqref{E:eq1}  that
the condition  $L(u,v)= M(u,v)=N(u,v) = 0$,
$(u,v) \in \mathcal D$ characterizes points at which
the space $\{\sigma(x,y):  x, y \in T_pM^2_1\}$  is one-dimensional.
We call such points  \emph{flat points} of the surface since they are analogous to flat points
in the theory of surfaces in $\R^3$.
In \cite{Lane} and \cite{Little} such points are called inflection points.
E. Lane has shown that every point of a surface is an inflection point
if and only if the surface is either developable or lies in a 3-dimensional space  \cite{Lane}.
So, further we consider surfaces free of flat points, i.e. we assume that $(L, M, N) \neq (0,0,0)$.

\vskip 2mm
The second fundamental form $II$ determines a map of Weingarten-type
$\gamma:T_{p}M^2_1\rightarrow T_{p}M^2_1$ at any point of $M^2_1$
in the standard way:

$$\begin{array}{l}
\vspace{2mm}\gamma(z_{u})=\gamma_{1}^{1}z_{u}+\gamma_{1}^{2}z_{v},\\
\vspace{2mm}\gamma(z_{v})=\gamma_{2}^{1}z_{u}+\gamma_{2}^{2}z_{v},
\end{array}$$
where
$${\displaystyle {\gamma_{1}^{1}=\frac{FM-GL}{EG-F^{2}},\quad\gamma_{1}^{2}=
\frac{FL-EM}{EG-F^{2}}},\quad{\displaystyle {\gamma_{2}^{1}=\frac{FN-GM}{EG-F^{2}},\quad\gamma_{2}^{2}=
\frac{FM-EN}{EG-F^{2}}}.}}$$
The linear map $\gamma$ is invariant under changes of the parameters
of the surface and changes of the normal frame field.  In general, $\gamma$ is not diagonalizable.

As in the classical differential geometry of surfaces the map $\gamma$ generates the following invariants:

$$k:=\det\gamma=\frac{LN-M^{2}}{EG-F^{2}},\qquad\varkappa:=-\frac{1}{2}\,{\rm tr}\,\gamma=
\frac{EN+GL-2FM}{2(EG-F^{2})}$$
\vskip 2mm
The functions $k$ and $\varkappa$ are invariant under changes of the parameters of the surface and
changes of the normal frame field.  Next we shall prove the following:

\begin{prop}
The function $\varkappa$ is the curvature
of the normal connection of the surface $M^2_1$.
\end{prop}

\noindent \emph{Proof:}
The curvature tensor $R^{\bot}$ of the normal connection $D$ is
given by
$$R^{\bot}(x,y)n=D_{x}D_{y}n-D_{y}D_{x}n-D_{[x,y]}n,$$
where $x$, $y$ are tangent vector fields and $n$ is a normal vector field of $M^2_1$.
The curvature of the normal connection at a point $p\in M^2_1$ is
defined by $\langle R^{\bot}(x,y)n_{2},n_{1}\rangle,$ where $\{x,y,n_{1},n_{2}\}$
is a right oriented orthonormal quadruple.

Without loss of generality we assume that $F = 0$ and take $x, y$ to be the unit vector fields
in the direction of  $z_u$ and $z_v$, i.e. $x=\ds{\frac{z_u}{\sqrt{E}}}$, $y=\ds{\frac{z_v}{\sqrt{-G}}}$.
Let $\{n_{1},n_{2}\}$ be a normal frame field  such that
 $\left\langle n_{1},n_{1}\right\rangle =1$, $\left\langle n_{1},n_{2}\right\rangle =0$,
$\left\langle n_{2},n_{2}\right\rangle =-1$.
Denote by $A_1$ (resp. $A_2$) the shape operator  corresponding to $n_1$ (resp. $n_2$).

Since the curvature tensor $R'$ of the connection $\nabla'$ is zero,
we have
$$\nabla'_{x}\nabla'_{y}n_{1}-\nabla'_{y}\nabla'_{x}n_{1}-\nabla'_{[x,y]}n_{1}=0.$$
Therefore the tangent component and the normal component of $R'(x,y)n_{1}$
are both zero. The normal component is $D_{x}D_{y}n_{1}-D_{y}D_{x}n_{1}-D_{[x,y]}n_{1}-
\sigma(x,A_{1}(y))+\sigma(y,A_{1}(x)).$
Hence,
$$D_{x}D_{y}n_{1}-D_{y}D_{x}n_{1}-D_{[x,y]}n_{1}=\sigma(x,A_{1}(y))-\sigma(y,A_{1}(x)).$$
The left-hand side of the last equality is $R^{\bot}(x,y)n_{1}$. Then
\begin{equation}\label{E:Eq2}
\langle R^{\bot}(x,y)n_{1},n_{2}\rangle=\langle\sigma(x,A_{1}(y)),n_{2}\rangle-\langle\sigma(y,A_{1}(x)),n_{2}\rangle=
\langle(A_{2}\circ A_{1}-A_{1}\circ A_{2})(y),x\rangle.
\end{equation}

Further we consider the operator $A_{2}\circ A_{1}-A_{1}\circ A_{2}$.
For the second fundamental tensor $\sigma$ we have:
\begin{equation} \label{E:Eq2-b}
\begin{array}{l}
\vspace{2mm}
\sigma(x,x)=\displaystyle{\frac{c_{11}^{1}}{E}\;\;n_{1}\;-\;\frac{c_{11}^{2}}{E}\;\;n_{2}},\\
\vspace{2mm}\sigma(x,y)=\displaystyle{\frac{c_{12}^{1}}{\sqrt{-EG}}\;n_{1}-\frac{c_{12}^{2}}{\sqrt{-EG}}\;n_{2}}\\
\vspace{2mm}\sigma(y,y)=\displaystyle{-\frac{c_{22}^{1}}{G}\;\;n_{1}\;+\;\;\frac{c_{22}^{2}}{G}\;\;n_{2}}
\end{array}
\end{equation}

Using the last formulas we calculate the vector fields $A_1(x)$, $A_1(y)$, $A_2(x)$, $A_2(y)$ and obtain:
\begin{equation} \label{E:Eq3}
\begin{array}{ll}
\vspace{2mm}A_{1}(x)={\displaystyle {\frac{c_{11}^{1}}{E}\;x\;-\;\frac{c_{12}^{1}}{\sqrt{-EG}}\;y},\qquad} & A_{2}(x)=
{\displaystyle {\frac{c_{11}^{2}}{E}\;x\;-\;\frac{c_{12}^{2}}{\sqrt{-EG}}\;y},}\\
\vspace{2mm}A_{1}(y)={\displaystyle {\frac{c_{12}^{1}}{\sqrt{-EG}}\;x+\frac{c_{22}^{1}}{G}\;y,}\qquad} & A_{2}(y)=
{\displaystyle {\frac{c_{12}^{2}}{\sqrt{-EG}}\;x+\frac{c_{22}^{2}}{G}\;y.}}
\end{array}
\end{equation}
Taking in mind \eqref{E:Eq3} we get:
$$\begin{array}{ll}
\vspace{4mm}(A_{2}\circ A_{1}-A_{1}\circ A_{2})(x)=
\ds{-\left(\frac{c_{11}^{1}c_{12}^{2}-c_{11}^{2}c_{12}^{1}}{E\sqrt{-EG}}+
\frac{c_{12}^{1}c_{22}^{2}-c_{12}^{2}c_{22}^{1}}{G\sqrt{-EG}}\right)\, y}=\ds{-\frac{EN+GL}{2EG}\,y;}\\
\vspace{2mm}(A_{2}\circ A_{1}-A_{1}\circ A_{2})(y)=\ds{-\left(\frac{c_{11}^{1}c_{12}^{2}-c_{11}^{2}c_{12}^{1}}
{E\sqrt{-EG}}+\frac{c_{12}^{1}c_{22}^{2}-c_{12}^{2}c_{22}^{1}}{G\sqrt{-EG}}\right) \,x}=-\ds{\frac{EN+GL}{2EG}\,x.}
\end{array}$$
Hence
$$\begin{array}{l}
\vspace{2mm}(A_{2}\circ A_{1}-A_{1}\circ A_{2})(x) = -\varkappa \,y;\\
\vspace{2mm}(A_{2}\circ A_{1}-A_{1}\circ A_{2})(y) = -\varkappa \,x.
\end{array}$$
The last equalities imply that:
\begin{equation}\label{E:Eq4}
\begin{array}{l}
\vspace{2mm}
\langle(A_{2}\circ A_{1}-A_{1}\circ A_{2})(x),y\rangle = \varkappa;\\
\vspace{2mm}
\langle(A_{2}\circ A_{1}-A_{1}\circ A_{2})(y),x\rangle = -\varkappa.
\end{array}
\end{equation}

Note that $A_2 \circ A_1 - A_1 \circ A_2$ is an invariant
skew-symmetric operator in the tangent space, i.e. it does not
depend on the choice of the orthonormal tangent frame field
$\{x,y\}$.

Finally, from \eqref{E:Eq2} and \eqref{E:Eq4} we get that
$\langle R^{\bot}(x,y)n_{1},n_{2}\rangle = -\varkappa$.
Consequently,
$$\langle R^{\bot}(x,y)n_{2},n_{1}\rangle = \varkappa.$$
\qed

\vskip 3mm
The second fundamental form determines  conjugate tangents at a point $p$ of $M^2_1$
in the same way as in the classical differential geometry of surfaces.

\begin{defn} \label{D:conjugate}
Two tangents $g_1: X_1 = \lambda_1 z_u + \mu_1 z_v$ and $g_2: X_2
= \lambda_2 z_u + \mu_2 z_v$ are said to be \emph{conjugate
tangents},   if
$$L\lambda_1 \lambda_2 + M (\lambda_1 \mu_2 +\lambda_2 \mu_1) + N\mu_1 \mu_2 = 0.$$
\end{defn}

Asymptotic tangents and principal tangents can be defined in the standard way:

\begin{defn} \label{D:asymptotic tangent}
A tangent $g: X = \lambda z_u + \mu z_v$ is said to be
\emph{asymptotic}, if it is self-conjugate.
\end{defn}

\begin{defn} \label{D:principal tangent}
A tangent $g: X = \lambda z_u + \mu z_v$ is said to be
\emph{principal}, if it is perpendicular to its conjugate.
\end{defn}

The equation of the asymptotic tangents at a point $p
\in M^2_1$ is
$$L\lambda^2 + 2M \lambda \mu + N\mu^2 = 0.$$

The equation of the principal tangents at a point $p
\in M^2_1$ is
\begin{equation} \label{E:eq2-aa}
(EM-FL) \lambda^2 + (EN-GL) \lambda \mu + (FN-GM) \mu^2 = 0.
\end{equation}

A line $c: u=u(q), \; v=v(q); \; q\in J \subset \R$ on $M^2_1$ is
said to be an \emph{asymptotic line}, respectively a
\textit{principal line}, if its tangent at any point is
asymptotic, respectively  principal.

In the theory of surfaces in $\E^4$ and the theory of spacelike surfaces in $\E^4_1$  equation
\eqref{E:eq2-aa} always has solutions
since the discriminant of \eqref{E:eq2-aa} is greater or equal to zero. So, at each point of
a surface in $\E^4$ (or a spacelike surface in $\E^4_1$)
there exist principal tangents.
For a Lorentz surface in $\E^4_2$ the existence of solutions of equation \eqref{E:eq2-aa} depends on
the sign of the invariant $\varkappa^2 - k$.
Indeed, it can easily be seen that the discriminant $D$ of \eqref{E:eq2-aa} is expressed as
$$D = 4(EG - F^2)^2 (\varkappa^2 - k).$$

In the case $\varkappa^2 - k > 0$ at each point of the surface there exist two principal tangents.
This case corresponds to the case when the  map  $\gamma$ is diagonalizable.
If   $\varkappa^2 - k > 0$ we can assume that the parametric lines of the surface are principal.
It is clear that  $M^2_1$ is parameterized by principal lines if and only if $F=0, \,\, M=0.$

 Further we study Lorentz surfaces in $\E^4_2$ for which $\varkappa^2 - k > 0$ at each point.

\section{Lorentz surfaces free of flat points} \label{S:Fund-Theorem}

\subsection{Minimal Lorentz surfaces}
Let $M^2_1$  be a Lorentz surface free of flat points, i.e. $(L, M, N) \neq (0,0,0)$.
Using the terminology from the classical differential geometry of surfaces, we call a point
$p \in M^2_1$ \textit{umbilical}
if the coefficients of the first and the second fundamental forms at $p$  are proportional,
i.e. $L = \rho \,E, \; M = \rho\, F, \; N = \rho\, G$ for some $\rho \in \R$.

The normal mean curvature vector field of the surface is $\ds{H = \frac{1}{2}\,{\rm tr} \sigma =
\frac{1}{2}\, (\sigma(x, x) -\sigma(y,y))}$.
Recall that  $M^2_1$ is said to be \textit{minimal} if the mean curvature
vector $H =0$. We shall prove the following characterization of  minimal Lorentz surfaces in $\E^4_2$.

\vskip 2mm
\begin{prop}\label{P:minimal}
Let $M^2_1$ be a Lorentz surface in $\E^4_2$ free of flat points.
Then $M^2_1$ is minimal if and only if  $M^2_1$ consists of umbilical points.
\end{prop}

\vskip 2mm \noindent \emph{Proof:}
Without loss of generality we assume that $F = 0$ and take $x, y$ to be the unit vector fields defined
by $x=\ds{\frac{z_u}{\sqrt{E}}}$, $y=\ds{\frac{z_v}{\sqrt{-G}}}$.
It follows from formulas \eqref{E:Eq2-b} that $\sigma(x,x) - \sigma(y,y) = 0$ if and only if
$$\left(\frac{c_{11}^{1}}{E} + \frac{c_{22}^{1}}{G} \right) n_1 - \left(\frac{c_{11}^{2}}{E} +
\frac{c_{22}^{1}}{G} \right) n_2 = 0,$$
or equivalently,
\begin{equation}\label{E:Eq3-a}
\begin{array}{l}
\vspace{2mm}
c_{22}^1 = \ds{-\frac{G}{E} \,c_{11}^1};\\
\vspace{2mm}
c_{22}^2 = \ds{-\frac{G}{E} \,c_{11}^2}.
\end{array}
\end{equation}

Now, if $H = 0$ then \eqref{E:Eq3-a} implies
$$\Delta_2 = 0;\quad \quad \frac{\Delta_3}{G}=\frac{\Delta_1}{E}.$$
Therefore
$$L = \rho\, E; \quad M = \rho \,F; \quad N = \rho \,G,$$
where $\rho$ is a function on $M^2_1$. Hence, all points of  $M^2_1$ are umbilical.

Conversely, if $L = \rho\, E; \; M = \rho \,F; \; N = \rho \,G, \; \rho \neq 0$, then
the condition $F=0$ implies that $M=0$. Hence,
$\displaystyle{\left|%
\begin{array}{cc}
\vspace{2mm}
  c_{11}^1 & c_{22}^1 \\
  c_{11}^2 & c_{22}^2 \\
\end{array}%
\right|=0}$, and so there exists a function $\widetilde{\rho}$ such that
$c_{22}^1=\widetilde{\rho} c_{11}^1$,
$c_{22}^2=\widetilde{\rho} c_{11}^2$. Further, the equality
$\displaystyle{\frac{L}{E}=\frac{N}{G}}$ implies that
$\displaystyle{\widetilde{\rho}=-\frac{G}{E}}$. Hence, equalities \eqref{E:Eq3-a} hold.
Consequently, $\tr \, \sigma = 0$, i.e. $H=0$.

\qed

\vskip 3mm

The meaning of the above proposition is that in $\E^4_2$  the Lorentz surfaces consisting of
umbilical points  are exactly the minimal surfaces. The complete classification of minimal surfaces in
an arbitrary indefinite pseudo-Eucliden space $\E^m_s$  is given by B.-Y. Chen in the next theorem.

\vskip 2mm
\begin{thm}\label{T:minimal} \cite{Chen3}
A Lorentz surface in a pseudo-Euclidean m-space $\E^m_s$ is minimal if and only if,
locally the surface is parametrized by
$$z(u,v) = \alpha(u) + \beta(v),$$
where $\alpha$ and $\beta$ are null curves satisfying $\left\langle \alpha'(u), \beta'(v) \right\rangle \neq 0$.
\end{thm}

\vskip 3mm
Further, we consider surfaces free of umbilical (minimal) points.
Note that if $p \in M^2_1$ is a non-umbilical point and $\varkappa^2 - k > 0$, then there exist exactly
two principal tangents passing through $p$.
We assume that
 $M^2_1$  is parameterized by principal lines, i.e. $F = 0$ and $M=0$, and suppose that the principal tangents
 $z_u$ and $z_v$  are spacelike and
 timelike, respectively.
 Consider the unit tangent vector fields $x$ and $y$  defined by $x=\ds{\frac{z_u}{\sqrt{E}}}$,
 $y=\ds{\frac{z_v}{\sqrt{-G}}}$, i.e. $x$ and $y$ are collinear with the principal directions.
The condition $M=0$ implies that $\sigma(x, x)$ and $\sigma(y, y)$ are collinear. So, there exists
a geometrically determined normal vector field $n$ such that
$$\begin{array}{l}
\vspace{2mm} \sigma(x, x)=\nu_1 n\\
\vspace{2mm} \sigma(y, y)=\nu_2 n,
\end{array}$$
where $\nu_1$ and $\nu_2$ are functions on $M^2_1$. Hence,  the mean curvature vector field $H$ is expressed as follows:
$$H=\frac{\nu_1-\nu_2}{2}\,n.$$

We have the following possibilities for the mean
curvature vector field:
\begin{itemize}
\item
$H$ is \textit{spacelike}, i.e. $\langle H,H \rangle >0$;

\item
$H$ is \textit{timelike}, i.e. $\langle H,H \rangle <0$;

\item
$H$ is \textit{lightlike}, i.e. $\langle H,H \rangle =0$.
\end{itemize}

If the mean curvature vector is lightlike at each point, i.e. $H \neq 0$ and $\langle H, H \rangle =0$,
the surface $M^2_1$ is \textit{quasi-minimal}.
In what follows we study surfaces free of minimal points and such that the mean curvature vector $H$ is
either spacelike or timelike at each point. We call such surfaces \textit{Lorentz surfaces of general type}.

\subsection{Lorentz surfaces whose mean curvature vector at each point is a
non-zero spacelike vector}\label{S:spacelikeH}

In this subsection we consider the case  $\left\langle  H, H  \right\rangle > 0$.
In order to obtain a geometrically determined frame field at each point $p$ of $M^2_1$, we denote by
$b$ the unit normal vector field defined by
$b=\ds{\frac{H}{{\sqrt{\langle H, H \rangle}}}}$. Note that $\langle b, b\rangle=1$ and $b$
is collinear with $\sigma(x, x)$ and
$\sigma(y, y)$. Further we consider the unit normal vector field $l$ such that $\{x, y, b, l\}$ is a
positively oriented orthonormal frame field in $\E_2^4$. So,   $\langle l, l\rangle=-1$. In such
a way we obtain a geometrically determined moving frame field at each point of the surface.
With respect to this frame field the following formulas are true:
$$\begin{array}{l}
\vspace{2mm} \sigma(x, x)=\nu_1 \,b;\\
\vspace{2mm} \sigma(x, y)=\lambda\, b-\mu\, l;\\
\vspace{2mm} \sigma(y, y)=\nu_2\, b,
\end{array}$$
where $\nu_1, \nu_2, \lambda$, $\mu$ are functions on $M^2_1$ determined by the geometric frame field
as follows: $\nu_1=\langle \sigma(x, x), b\rangle$, $\nu_2=\langle \sigma(y, y), b\rangle$,
$\lambda=\langle \sigma(x, y), b\rangle$, $\mu=\langle \sigma(x, y), l\rangle$.

Now with respect to the geometric frame field $\{x, y, b, l\}$ we have the following Frenet-type
derivative formulas of $M^2_1$:
\begin{equation}\label{E:eq3}
\begin{array}{ll}
\vspace{2mm}\nabla'_{x}x=\quad\quad\;-\gamma_{1}\,y+\;\nu_{1}\,b, & \qquad\quad\nabla'_{x}b=
-\nu_{1}\,x+\lambda\,y\quad\,\quad-\beta_{1}\,l,\\
\vspace{2mm}\nabla'_{x}y=-\gamma_{1}\,x\quad\quad\;+\;\lambda\,b-\mu\;l, & \qquad\quad\nabla'_{y}b=
-\lambda\,x+\nu_{2}\,y\quad\quad\,-\beta_{2}\,l,\\
\vspace{2mm}\nabla'_{y}x=\quad\quad\;-\gamma_{2}\,y+\;\lambda\,b-\mu\;l, & \qquad\quad\nabla'_{x}l=
\quad\quad\quad\ \mu\;y-\beta_{1}\,b,\\
\vspace{2mm}\nabla'_{y}y=\;-\gamma_{2}\,x\quad\quad+\nu_{2}\,b, & \qquad\quad\nabla'_{y}l=
-\mu\;x\quad\quad\;-\beta_{2}\,b,
\end{array}
\end{equation}
where $\gamma_1=-y(\ln\sqrt{E}) =\langle\nabla'_xx, y\rangle $, $\gamma_2=-x(\ln\sqrt{-G}) =
\langle\nabla'_yx, y\rangle $, $\beta_1=\langle\nabla'_{x}b, l\rangle$ and $\beta_2=\langle\nabla'_{y}b, l\rangle.$

The mean curvature vector field is expressed as
$$H=\frac{\nu_1-\nu_2}{2}\, b.$$

\subsection{Lorentz surfaces whose mean curvature vector at each point is a
timelike vector}\label{S:timelikeH}

In this subsection we consider the case  $\left\langle  H, H  \right\rangle < 0$.
Now we set the unit normal vector field $b$ to be $b = \ds{-\frac{H}{{\sqrt{-\langle H, H \rangle}}}}$.
In this case we have $\langle b, b\rangle=-1$. Taking the normal unit vector field $l$ in such a way
that $\{x, y, l, b\}$ be a positively oriented orthonormal frame field in
$\E_2^4$, we get a geometrically determined orthonormal frame field at each point of the surface.
Note that $\langle l, l \rangle = 1$.
The  Frenet-type derivative formulas  of the surface look as follows:
\begin{equation}\label{E:eq4}
\begin{array}{ll}
\vspace{2mm}\nabla'_{x}x=\quad\quad\;-\gamma_{1}\,y\quad\quad\ -\nu_{1}\,b, & \qquad\quad\nabla'_{x}l=
\quad\quad\quad\quad\mu\;y\quad\quad+\beta_{1}\,b,\\
\vspace{2mm}\nabla'_{x}y=-\gamma_{1}\,x\quad\quad\;+\;\mu\,l-\lambda\;b, & \qquad\quad\nabla'_{y}l=
-\mu\;x\quad\quad\;\quad\quad\quad+\beta_{2}\,b,\\
\vspace{2mm}\nabla'_{y}x=\quad\quad\;-\gamma_{2}\,y+\;\mu\,l-\lambda\;b, & \qquad\quad\nabla'_{x}b=
-\nu_{1}\,x+\lambda\,y+\beta_{1}\,l,\\
\vspace{2mm}\nabla'_{y}y=\;-\gamma_{2}\,x\quad\quad\quad\quad\ -\nu_{2}\,b, & \qquad\quad\nabla'_{y}b=
-\lambda\,x+\nu_{2}\,y+\beta_{2}\,l.\\
\end{array}
\end{equation}
where $\nu_1$, $\nu_2$, $\lambda$, $\mu$, $\gamma_1$, $\gamma_2$, $\beta_1$, $\beta_2$ are determined
by the geometric moving frame field in the same way as in the previous case. In this case
the mean curvature vector field $H$ is expressed by the formula
$$H=-\frac{\nu_1-\nu_2}{2} \,b.$$

\subsection{Fundamental theorem for Lorentz surfaces of general type}

The general fundamental existence and uniqueness theorems for
submanifolds of pseudo-Riemannian manifolds are formulated in
terms of tensor fields and connections on vector bundles (e.g.
\cite{Chen2}, Theorem 2.4 and Theorem 2.5). In \cite{GM3} we
formulated and proved a fundamental theorem in terms of
the geometric functions of a spacelike surface  in $\R^4_1$ whose mean
curvature vector at any point is a non-zero spacelike vector or
timelike vector.  In \cite{GM4} we proved the fundamental existence and
uniqueness theorem for marginally trapped surfaces in terms of
their geometric functions.
Here we shall formulate and prove the fundamental existence and
uniqueness theorem for Lorentz surfaces in $\E^4_2$ of the general class.
This theorem is a special case of the
general fundamental theorem but in the present form it is  more
appropriate and easier to apply.

The Frenet-type formulas given in \eqref{E:eq3} and \eqref{E:eq4} for the case of spacelike and timelike
mean curvature vector field, respectively,
can be unified as follows:
\begin{equation}\label{E:eq5}
\begin{array}{ll}
\vspace{2mm}\nabla'_{x}x=\quad\quad\;-\gamma_{1}\,y+\varepsilon\:\nu_{1}\,b, & \qquad\quad\nabla'_{x}b=
-\nu_{1}\,x+\lambda\,y\quad\,\quad\quad-\varepsilon\:\beta_{1}\,l,\\
\vspace{2mm}\nabla'_{x}y=-\gamma_{1}\,x\quad\quad\;+\varepsilon\:\lambda\,b-\varepsilon\:\mu\;l, & \qquad\quad\nabla'_{y}b=
-\lambda\,x+\nu_{2}\,y\quad\quad\quad\,-\varepsilon\:\beta_{2}\,l,\\
\vspace{2mm}\nabla'_{y}x=\quad\quad\;-\gamma_{2}\,y+\varepsilon\:\lambda\,b-\varepsilon\:\mu\;l & \qquad\quad\nabla'_{x}l=
\quad\quad\quad\ \mu\;y-\varepsilon\:\beta_{1}\,b,\\
\vspace{2mm}\nabla'_{y}y=\;-\gamma_{2}\,x\quad\quad+\varepsilon\:\nu_{1}\,b, & \qquad\quad\nabla'_{y}l=
-\mu\;x\quad\quad\;-\varepsilon\:\beta_{2}\,b,
\end{array}
\end{equation}
where $\varepsilon=1$ in the case $\langle H, H \rangle >0$ and $\varepsilon=-1$ in the case $\langle H, H \rangle <0$.
In both cases
$\ds{b = \frac{\varepsilon H}{\sqrt{\varepsilon \langle H, H \rangle}}}$; $\langle b, b \rangle = \varepsilon$;
$\langle l, l \rangle = -\varepsilon$; $\langle b, l \rangle = 0$.

\vskip 2mm
Taking into account that $R'(x, y, x)=0$, $R'(x, y, y)=0$, $R'(x, y, b)=0$, and using derivative formulas
\eqref{E:eq5} we get the following integrability conditions:

\begin{equation} \label{E:eq5-i}
\begin{array}{l}
\vspace{2mm}
2\mu\,\gamma_{2}-\varepsilon\,\lambda\,\beta_{1}+\varepsilon\,\nu_{1}\,\beta_{2}=x(\mu);\\
\vspace{2mm}
2\mu\,\gamma_{1}+\varepsilon\,\nu_{2}\,\beta_{1}-\varepsilon\,\lambda\,\beta_{2}=y(\mu);\\
\vspace{2mm}
2\lambda\,\gamma_{2}-\varepsilon\,\mu\,\beta_{1}-(\nu_{1}+\nu_{2})\,\gamma_{1}=x(\lambda)-y(\nu_{1});\\
\vspace{2mm}
2\lambda\,\gamma_{1}-\varepsilon\,\mu\,\beta_{2}-(\nu_{1}+\nu_{2})\,\gamma_{2}=-x(\nu_{2})+y(\lambda);\\
\vspace{2mm}
\gamma_{1}\,\beta_{1}-\gamma_{2}\,\beta_{2}+(\nu_{1}+\nu_{2})\,\mu=-x(\beta_{2})+y(\beta_{1});\\
\vspace{2mm}
\varepsilon(\lambda^{2}-\mu^{2}-\nu_{1}\,\nu_{2})=x(\gamma_{2})-y(\gamma_{1})+(\gamma_{1})^{2}-(\gamma_{2})^{2}.
\end{array}
\end{equation}

Having in mind that $x={\ds{\frac{z_{u}}{\sqrt{E}},\,y=\frac{z_{v}}{\sqrt{-G}}}}$, we can write
the integrability conditions in the following way:
$$\begin{array}{l}
\vspace{2mm}
2\mu\,\gamma_{2}-\varepsilon\,\lambda\,\beta_{1}+\varepsilon\,\nu_{1}\,\beta_{2}=
{\displaystyle {\frac{1}{\sqrt{E}}\,\mu_{u}};}\\
\vspace{2mm}
2\mu\,\gamma_{1}+\varepsilon\,\nu_{2}\,\beta_{1}-\varepsilon\,\lambda\,\beta_{2}=
{\displaystyle {\frac{1}{\sqrt{-G}}\,\mu_{v}};}\\
\vspace{2mm}
2\lambda\,\gamma_{2}-\varepsilon\,\mu\,\beta_{1}-(\nu_{1}+\nu_{2})\,\gamma_{1}=
{\displaystyle {\frac{1}{\sqrt{E}}\,\lambda_{u}-\frac{1}{\sqrt{-G}}\,(\nu_{1})_{v}};}\\
\vspace{2mm}
2\lambda\,\gamma_{1}-\varepsilon\,\mu\,\beta_{2}-(\nu_{1}+\nu_{2})\,\gamma_{2}=
{\displaystyle {-\frac{1}{\sqrt{E}}\,(\nu_{2})_{u}+\frac{1}{\sqrt{-G}}\,\lambda_{v}};}\\
\vspace{2mm}
\gamma_{1}\,\beta_{1}-\gamma_{2}\,\beta_{2}+(\nu_{1}+\nu_{2})\,\mu=
{\displaystyle {-\frac{1}{\sqrt{E}}\,(\beta_{2})_{u}+\frac{1}{\sqrt{-G}}\,(\beta_{1})_{v}};}\\
\vspace{2mm}
\varepsilon(\lambda^{2}-\mu^{2}-\nu_{1}\,\nu_{2})=
{\ds{\frac{1}{\sqrt{E}}\,(\gamma_{2})_{u}-\frac{1}{\sqrt{-G}}\,(\gamma_{1})_{v}+(\gamma_{1})^{2}-(\gamma_{2})^{2}}.}\\
\end{array}$$

It is clear that if $\mu_u\,\mu_v \neq 0$, then we can express the functions $\sqrt{E}$ and $\sqrt{-G}$
in the following way:
$$\begin{array}{l}
\vspace{2mm} \sqrt{E}={\ds\frac{\mu_u}{2\mu\,\gamma_{2}-\varepsilon\,\lambda\,\beta_{1}+
\varepsilon\,\nu_{1}\,\beta_{2}}};\\
\vspace{2mm} \sqrt{-G}={\ds\frac{\mu_v}{2\mu\,\gamma_{1}+\varepsilon\,\nu_{2}\,\beta_{1}-
\varepsilon\,\lambda\,\beta_{2}}}.
\end{array}$$
The condition $\mu_u \,\mu_v \neq 0$ is equivalent to
$(2\mu\,\gamma_{2}-\varepsilon\,\lambda\,\beta_{1}+\varepsilon\,\nu_{1}\,\beta_{2}) (2\mu\,\gamma_{1}+
\varepsilon\,\nu_{2}\,\beta_{1}-\varepsilon\,\lambda\,\beta_{2}) \neq 0$.

\vskip 2mm

Now we shall prove the following  Bonnet-type theorem for
Lorentz surfaces in $\E^4_2$.

\begin{thm}\label{T:Fund-Theorem}
Let $\gamma_{1},\,\gamma_{2},\,\nu_{1},\,\nu_{2},\,\lambda,\,\mu,\,\beta_{1},\beta_{2}$
be smooth functions, defined in a domain $\mathcal{D},\,\,\mathcal{D}\subset{\R}^{2}$,
and satisfying the conditions
\begin{equation}\label{E:eq6}
\begin{array}{l}
\vspace{2mm}
{\displaystyle {\frac{\mu_{u}}{2\mu\,\gamma_{2}-\varepsilon\,\lambda\,\beta_{1}+\varepsilon\,\nu_{1}\,\beta_{2}}}>0;}\\
\vspace{2mm}
{\displaystyle {\frac{\mu_{v}}{2\mu\,\gamma_{1}+\varepsilon\,\nu_{2}\,\beta_{1}-\varepsilon\,\lambda\,\beta_{2}}}>0;}\\
\vspace{2mm}
-\gamma_{1}\sqrt{E}\sqrt{-G}=(\sqrt{E})_{v};\\
\vspace{2mm}
-\gamma_{2}\sqrt{E}\sqrt{-G}=(\sqrt{-G})_{u};\\
\vspace{2mm}
2\lambda\,\gamma_{2}-\varepsilon\,\mu\,\beta_{1}-(\nu_{1}+\nu_{2})\,\gamma_{1}=
{\displaystyle {\frac{1}{\sqrt{E}}\,\lambda_{u}-\frac{1}{\sqrt{-G}}\,(\nu_{1})_{v}}};\\
\vspace{2mm}
2\lambda\,\gamma_{1}-\varepsilon\,\mu\,\beta_{2}-(\nu_{1}+\nu_{2})\,\gamma_{2}=
{\displaystyle {-\frac{1}{\sqrt{E}}\,(\nu_{2})_{u}+\frac{1}{\sqrt{-G}}\,\lambda_{v}};}\\
\vspace{2mm}
\gamma_{1}\,\beta_{1}-\gamma_{2}\,\beta_{2}+(\nu_{1}+\nu_{2})\,\mu=
{\displaystyle {-\frac{1}{\sqrt{E}}\,(\beta_{2})_{u}+\frac{1}{\sqrt{-G}}\,(\beta_{1})_{v}}};\\
\varepsilon(\lambda^{2}-\mu^{2}-\nu_{1}\,\nu_{2})={\ds{\frac{1}{\sqrt{E}}\,(\gamma_{2})_{u}-\frac{1}
{\sqrt{-G}}\,(\gamma_{1})_{v}+(\gamma_{1})^{2}-(\gamma_{2})^{2}}}.\\
\end{array}
\end{equation}
where $\vspace{2mm} \sqrt{E}={\ds\frac{\mu_u}{2\mu\,\gamma_{2}-\varepsilon\,\lambda\,\beta_{1}+
\varepsilon\,\nu_{1}\,\beta_{2}}},
\sqrt{-G}={\ds\frac{\mu_v}{2\mu\,\gamma_{1}+\varepsilon\,\nu_{2}\,\beta_{1}-\varepsilon\,\lambda\,\beta_{2}}}$.
Let $\{x_{0},\,y_{0},\,b_{0},\,l_{0}\}$ be an orthonormal frame at
a point $p_{0}\in\E^4_2$ (with $\langle b_0, b_0 \rangle = \varepsilon$;
$\langle l_0, l_0 \rangle = -\varepsilon$; $\langle b_0, l_0 \rangle = 0$). Then there exist a subdomain
${\mathcal{D}}_{0}\subset\mathcal{D}$
and a unique Lorentz surface $M^2_1:z=z(u,v),\,\,(u,v)\in{\mathcal{D}}_{0}$,
whose mean curvature vector at any point is a non-zero spacelike vector if $\varepsilon = 1$, or timelike
vector if $\varepsilon = -1$.
Moreover, $M^2_1$ passes through $p_{0}$, the functions
$\gamma_{1},\,\gamma_{2},\,\nu_{1},\,\nu_{2},\,\lambda,\,\mu,\,\beta_{1},\beta_{2}$ are
the geometric functions of $M^2_1$ and $\{x_{0},\,y_{0},\,b_{0},\,l_{0}\}$
is the geometric frame of $M^2_1$ at  point $p_{0}$.
\end{thm}

\vskip 2mm
\begin{proof}
We consider the following system of partial
differential equations for the unknown vector functions $x=x(u,v),\,y=y(u,v),\,b=b(u,v),\,l=l(u,v)$
in $\E_{2}^{4}$:
\begin {equation}\label{E:eq7}
\begin{array}{ll}
\vspace{2mm}x_{u}=(-\gamma_{1}\,y+\varepsilon\,\nu_{1}\,b)\sqrt{E} & \qquad x_{v}=
(-\gamma_{2}\,y+\varepsilon\,\lambda\,b-\varepsilon\,\mu\,l)\sqrt{-G}\\
\vspace{2mm}y_{u}=(-\gamma_{1}\,x+\varepsilon\,\lambda\,b-\varepsilon\,\mu\,l)\sqrt{E} & \qquad y_{v}=
(-\gamma_{2}\,x+\varepsilon\,\nu_{2}\,b)\sqrt{-G}\\
\vspace{2mm}b_{u}=(-\nu_{1}\,x+\lambda\,y-\varepsilon\,\beta_{1}\,l)\sqrt{E} & \qquad b_{v}=(-\lambda\,x+
\nu_{2}\,y-\varepsilon\,\beta_{2}\,l)\sqrt{-G}\\
\vspace{2mm}l_{u}=(\mu\,y-\varepsilon\,\beta_{1}\,b)\sqrt{E} & \qquad l_{v}=
(-\mu\,x-\varepsilon\,\beta_{2}\,b)\sqrt{-G}
\end{array}
\end{equation}
We denote
$$Z=\left(\begin{array}{c}
x\\
y\\
b\\
l
\end{array}\right);\quad
A=\sqrt{E}\left(\begin{array}{cccc}
0 & -\gamma_{1} & \varepsilon\,\nu_{1} & 0\\
-\gamma_{1} & 0 & \varepsilon\,\lambda & -\varepsilon\,\mu\\
-\nu_{1} & \lambda & 0 & -\varepsilon\,\beta_{1}\\
0 & \mu & -\varepsilon\,\beta_{1} & 0
\end{array}\right);$$

$$B=\sqrt{-G}\left(\begin{array}{cccc}
0 & -\gamma_{2} & \varepsilon\,\lambda & -\varepsilon\,\mu\\
-\gamma_{2} & 0 & \varepsilon\,\nu_{2} & 0\\
-\lambda & \nu_{2} & 0 & -\varepsilon\,\beta_{2}\\
-\mu & 0 & -\varepsilon\,\beta_{2} & 0
\end{array}\right).$$
Using  matrices $A$ and $B$ we can rewrite  system \eqref{E:eq7} in the form:
\begin {equation}\label{E:eq8}
\begin{array}{l}
\vspace{2mm}Z_{u}=A\,Z;\\
\vspace{2mm}Z_{v}=B\,Z.
\end{array}
\end{equation}

The integrability conditions of system \eqref{E:eq8} are
$$Z_{uv}=Z_{vu},$$
i.e.
\begin {equation}\label{E:eq9}
{\displaystyle {\frac{\partial a_{i}^{k}}{\partial v}-\frac{\partial b_{i}^{k}}
{\partial u}+\sum_{j=1}^{4}(a_{i}^{j}\,b_{j}^{k}-b_{i}^{j}\,a_{j}^{k})=0,\quad i,k=1,\dots,4,}}
\end {equation}
where $a_{i}^{j}$ and $b_{i}^{j}$ are the elements of the matrices
$A$ and $B$. Using \eqref{E:eq6} we obtain that equalities \eqref{E:eq9} are fulfilled.
Hence, there exist a subset $\mathcal{D}_{1}\subset\mathcal{D}$
and unique vector functions $x=x(u,v),\,y=y(u,v),\,b=b(u,v),\,l=l(u,v),\,\,(u,v)\in\mathcal{D}_{1}$,
which satisfy system \eqref{E:eq7} and the initial conditions
$$x(u_{0},v_{0})=x_{0},\quad y(u_{0},v_{0})=y_{0},\quad b(u_{0},v_{0})=b_{0},\quad l(u_{0},v_{0})=l_{0}.$$

We shall prove that the vectors $x(u,v),\,y(u,v),\,b(u,v),\,l(u,v)$
form an orthonormal frame in $\E_{2}^{4}$ for each $(u,v)\in\mathcal{D}_{1}$.
Let us consider the following functions:
$$\begin{array}{lll}
\vspace{2mm}\varphi_{1}=\langle x,x\rangle-1; & \qquad\varphi_{5}=\langle x,y\rangle; & \qquad\varphi_{8}=
\langle y,b\rangle;\\
\vspace{2mm}\varphi_{2}=\langle y,y\rangle+1; & \qquad\varphi_{6}=\langle x,b\rangle; & \qquad\varphi_{9}=
\langle y,l\rangle;\\
\vspace{2mm}\varphi_{3}=\langle b,b\rangle-\varepsilon; & \qquad\varphi_{7}=\langle x,l\rangle; & \qquad\varphi_{10}=
\langle b,l\rangle;\\
\vspace{2mm}\varphi_{4}=\langle l,l\rangle+\varepsilon;
\end{array}$$
defined for  $(u,v)\in\mathcal{D}_{1}$.
Since $x(u,v),\,y(u,v),\,b(u,v),\,l(u,v)$
satisfy \eqref{E:eq7}, for the derivatives of $\varphi_{i}$ we obtain the following system:
$$\begin{array}{lll}
\vspace{2mm}{\displaystyle {\frac{\partial\varphi_{i}}{\partial u}=\alpha_{i}^{j}\,\varphi_{j}},}\\
\vspace{2mm}{\displaystyle {\frac{\partial\varphi_{i}}{\partial v}=\beta_{i}^{j}\,\varphi_{j}};}
\end{array}\qquad i=1,\dots,10,$$
where $\alpha_{i}^{j},\beta_{i}^{j},\,\,i,j=1,\dots,10$ are functions of $(u,v)\in\mathcal{D}_{1}$.
This is a linear system of partial differential equations
for the functions
$\varphi_i(u,v), \,\,i = 1, \dots, 10, \,\,(u,v) \in
\mathcal{D}_1$, satisfying the initial conditions $\varphi_i(u_0,v_0) = 0, \,\,i = 1,
\dots, 10$.
 Hence, $\varphi_{i}(u,v)=0,\,\,i=1,\dots,10$ for each $(u,v)\in\mathcal{D}_{1}$.
Consequently, the vector functions $x(u,v),\,y(u,v),\,b(u,v),\,l(u,v)$ form an orthonormal frame in
$\E_{2}^{4}$ for each $(u,v)\in\mathcal{D}_{1}$
(with  $\langle b, b \rangle = \varepsilon$;
$\langle l, l \rangle = -\varepsilon$; $\langle b, l \rangle = 0$).

Now we consider the following system of partial differential equations for the vector function $z=z(u, v)$:
\begin{equation}\label{E:eq10}
\begin{array}{l}
\vspace{2mm}z_{u}=\sqrt{E}\,x\\
\vspace{2mm}z_{v}=\sqrt{-G}\,y
\end{array}
\end{equation}
By the use of \eqref{E:eq6} and  \eqref{E:eq7} we get that the integrability conditions $z_{uv}=z_{vu}$
of system \eqref{E:eq10} are fulfilled.
Hence, there exist a subset $\mathcal{D}_{0}\subset\mathcal{D}_{1}$
and a unique vector function $z=z(u,v)$, defined for $(u,v)\in\mathcal{D}_{0}$
and satisfying $z(u_{0},v_{0})=p_{0}$.

Consequently, the surface $M^2_1:z=z(u,v),\,\,(u,v)\in\mathcal{D}_{0}$
satisfies the assertion of the theorem.

\end{proof}

\section{Basic classes of Lorentz surfaces characterized in terms of their geometric functions}

Further we shall give expressions for the Gauss curvature $K$ and the invariant functions $k$ and $\varkappa$.

We use that for pseudo-Riemannian submanifolds the Gauss curvature is given by the following formula:
$$K=\ds{\frac{\langle \sigma(x, x),\sigma(y, y)\rangle-\langle \sigma(x, y),\sigma(x, y)\rangle}
{\langle x, x\rangle\langle y, y\rangle-\langle x, y\rangle^2}}.$$
Since $\langle x, x \rangle = 1$, $\langle y, y \rangle = -1$, $\langle x, y \rangle = 0$,
from \eqref{E:eq5} we get the following expression for the Gauss curvature:
\begin{equation} \label{E:eq11}
K=\varepsilon\,(\lambda^2-\mu^2-\nu_1\nu_2).
\end{equation}

The curvature of the normal connection is expressed in terms of the geometric functions
$\mu$, $\nu_1$, and $\nu_2$ by the next formula:
\begin{equation} \label{E:eq12}
\varkappa = \mu (\nu_1+\nu_2).
\end{equation}

The invariant function $k$, generated by the Weingarten map $\gamma$ according to  the formula
$k=\det\gamma=\ds{\frac{LN-M^{2}}{EG-F^{2}}}$, is expressed as
\begin{equation} \label{E:eq13}
k=4\mu^2\nu_1\nu_2.
\end{equation}

Since we consider Lorentz surfaces for which $\varkappa^2 - k > 0$, equalities \eqref{E:eq12}
and \eqref{E:eq13} imply that
$\mu \neq 0$ and $\nu_1 \neq \nu_2$.

The next statements follow directly from formulas \eqref{E:eq11} and \eqref{E:eq12}.

\vskip 2mm
\begin{prop}\label{P:flat}
Let $M^2_1$ be a Lorentz surface of general type in $\E^4_2$.
Then $M^2_1$ is flat if and only if $\lambda^2-\mu^2 = \nu_1\nu_2$.
\end{prop}

\vskip 2mm
\begin{prop}\label{P:constant K}
Let $M^2_1$ be a Lorentz surface of general type in $\E^4_2$.
Then $M^2_1$ has constant Gauss curvature  if and only if $\lambda^2-\mu^2 - \nu_1\nu_2 = const$.
\end{prop}

\vskip 2mm
\begin{prop}\label{P:flat normal}
Let $M^2_1$ be a Lorentz surface of general type in $\E^4_2$.
Then $M^2_1$ is of flat normal connection  if and only if $\nu_1 + \nu_2 = 0$.
\end{prop}

\vskip 2mm
\begin{prop}\label{P:constant normal}
Let $M^2_1$ be a Lorentz surface of general type in $\E^4_2$.
Then $M^2_1$ has constant normal curvature $\varkappa$ if and only if $\mu (\nu_1+\nu_2) = const$.
\end{prop}

\vskip 2mm
Now we shall characterize some basic classes of surfaces satisfying special conditions on
the mean curvature vector field.
It follows from \eqref{E:eq5} that the mean curvature vector field $H$ is expressed by the formula:
\begin{equation} \label{E:eq14}
H = \frac{\varepsilon(\nu_1-\nu_2)}{2}\; b.
\end{equation}
Hence, the length of the mean curvature vector field is:
\begin{equation*} \label{E:eq15}
\|H\| = \varepsilon \sqrt{| \langle H, H \rangle|}= \varepsilon \frac{|\nu_1-\nu_2|}{2}.
\end{equation*}

\vskip 2mm
\begin{prop}\label{P:CMC}
Let $M^2_1$ be a Lorentz surface of general type in $\E^4_2$.
Then $M^2_1$ has non-zero constant mean curvature   if and only if $\nu_1 -\nu_2 = const \neq 0$.
\end{prop}

Using \eqref{E:eq14} and \eqref{E:eq5} we get that
\begin{equation*}\label{E:eq16}
\begin{array}{l}
\vspace{2mm}
D_x H = \ds{\varepsilon  x\left(\frac{\nu_1 -\nu_2}{2}\right)  b - \frac{\nu_1 -\nu_2}{2} \, \beta_1 \, l};\\
\vspace{2mm}
D_y H = \ds{\varepsilon  y\left(\frac{\nu_1 -\nu_2}{2}\right)  b - \frac{\nu_1 -\nu_2}{2} \, \beta_2 \, l}.
\end{array}
\end{equation*}
The last equalities imply that $H$ is parallel in the normal bundle (i.e. $D H = 0$ holds identically)
if and only if $\beta_1 = \beta_2 = 0$ and $\nu_1 -\nu_2 = const$.
So, the next statement holds true.

\vskip 2mm
\begin{prop}\label{P:CMC}
Let $M^2_1$ be a Lorentz surface of general type in $\E^4_2$.
Then $M^2_1$ has parallel mean curvature vector field  if and only if $\beta_1 = \beta_2 = 0$ and
$\nu_1 -\nu_2 = const$.
\end{prop}

\vskip 2mm
The geometric meaning of the invariant $\lambda$ is connected with the notion of allied mean
curvature vector field
defined by B.-Y. Chen for submanifolds of Riemannian manifolds.
Let $M$ be an $n$-dimensional submanifold of $(n+m)$-dimensional
Riemannian manifold $\widetilde{M}$ and $\xi$ be a normal vector
field of $M$. The \emph{allied
vector field} $a(\xi)$ of $\xi$ is defined by the formula
$$a(\xi)=\ds{\frac{\|\xi\|}{n}\sum_{k=2}^{m}\{\tr(A_{1} \circ A_{k})\}\xi_{k}},$$
where $\{\xi_{1}=\ds{\frac{\xi}{\|\xi\|}},\xi_{2}, \dots, \xi_{m}\}$ is an
orthonormal base of the normal space of $M$, and $A_{i}=A_{\xi_{i}},\,\,i=1,\dots,m$
is the shape operator with respect to $\xi_{i}$. In particular, the
allied vector field $a(H)$ of the mean curvature vector field $H$
is called the \emph{allied mean curvature vector field} of $M$
in $\widetilde{M}$ \cite{Chen1}.
B.-Y. Chen
defined  the $\mathcal{A}$-submanifolds to be those submanifolds
of $\widetilde{M}$ for which
 $a(H)$ vanishes identically.
The $\mathcal{A}$-submanifolds are also
called \emph{Chen submanifolds}. It is easy to see that minimal
submanifolds, pseudo-umbilical submanifolds and hypersurfaces are
Chen submanifolds. These Chen submanifolds are said to be trivial
$\mathcal{A}$-submanifolds.

The notion of allied mean curvature vector field is extended by S. Haesen and  M. Ortega
to the case when the normal space is a two-dimensional
Lorentz space \cite{Haesen-Ort-3}.

In the following proposition we characterize non-trivial  Chen surfaces in $\E^4_2$.
\vskip 2mm
\begin{prop}\label{P:Chen surface}
Let $M^2_1$ be a Lorentz surface of general type in $\E^4_2$.
Then $M^2_1$ is a non-trivial Chen surface if and only if $\lambda = 0$.
\end{prop}

\begin{proof}
Let $M^2_1$ be a Lorentz surface of general type in $\E^4_2$ and $\{b, l\}$ be the geometric
normal frame field defined in
Section \ref{S:Fund-Theorem}.
The allied mean curvature vector field is given by the formula
\begin{equation} \label{E:eq17}
a(H)=\ds{\frac{\|H\|}{2} \tr(A_b \circ A_l)\}l},
\end{equation}
where $A_b$ and $A_l$ are the shape operators corresponding to $b$ and $l$, respectively.
Using equalities \eqref{E:eq5}  we obtain
\begin{equation*}
\begin{array}{ll}
\vspace{2mm}
A_b(x) = \nu_{1}\,x -\lambda\,y; & \qquad A_l(x) = -\mu\,y;\\
\vspace{2mm}
A_b(y) = \lambda\,x - \nu_{2}\,y; & \qquad A_l(y) = \mu\,x.\\
\end{array}
\end{equation*}
Hence, $ \tr(A_b \circ A_l) = - 2\lambda \mu$. Now applying \eqref{E:eq17} and using  \eqref{E:eq14}, we get
$$a(H)=\ds{-\varepsilon \frac{|\nu_1-\nu_2|}{2} \lambda\,\mu \,l}.$$
Since $\mu \neq 0$ and $\nu_1 \neq \nu_2$, we conclude that  $a(H) = 0$ if and
only if $\lambda = 0$. This gives the geometric meaning of the
invariant $\lambda$:  $M^2_1$ is a non-trivial  Chen
surface if and only if the geometric function $\lambda$ is zero.

\end{proof}

\section{Lorentz surfaces with parallel normalized mean curvature vector field}\label{S:Parallel b}

In this section we shall focus our attention on a special class of
Lorentz surfaces in $\E^4_2$, namely surfaces with parallel normalized mean curvature vector field.

Let $M^2_1$ be a Lorentz surface of general type. In the previous section we proved that
the mean curvature vector field $H$ is parallel if and only if $\beta_1 = \beta_2 =0$ and $\nu_1 - \nu_2 = const$.
Now we shall consider the class of surfaces satisfying the conditions $\beta_1 = \beta_2 =0$.
It follows from \eqref{E:eq5} that $\beta_1 = \beta_2 = 0$ if and only if $D_xb = D_yb = 0$
(or equivalently, $D_xl = D_yl = 0$).
Since $b$ is a unit normal vector field in the direction of the mean curvature vector field,
the conditions  $\beta_1 = \beta_2 = 0$
characterize surfaces with parallel normalized mean curvature vector field. Hence, we have the following statement.

\vskip 2mm
\begin{prop}\label{P:Parallel b}
Let $M^2_1$ be a Lorentz surface of general type in $\E^4_2$.
Then $M^2_1$ has parallel normalized mean curvature vector field  if and only if $\beta_1 = \beta_2 = 0$.
\end{prop}

Note that the condition to have parallel normalized mean curvature vector field is weaker than
the condition to have parallel mean curvature vector field. In what follows we shall study Lorentz surfaces
with parallel normalized mean
curvature vector field, but not with parallel mean curvature vector field, i.e.
$\beta_1 = \beta_2 =0$, $\Vert H \Vert \neq const$. For these surfaces we shall introduce canonical parameters.

\begin{defn}\label{D:Canonocal parameters}
Let $M^2_1$ be a Lorentz surface of general type with parallel normalized mean
curvature vector field. The parameters $(u,v)$ of $M^2_1$ are said to be \textit{canonical}, if
$$E = \ds{\frac{1}{|\mu|}};\qquad F = 0; \qquad G = \ds{-\frac{1}{|\mu|}}.$$
\end{defn}

\begin{thm}\label{T:Canonocal parameters}
Each Lorentz surface of general type with  parallel normalized mean
curvature vector field in $\E^4_2$ locally admits  canonical parameters.
\end{thm}

\begin{proof}
Using the Gauss and Codazzi equations, from \eqref{E:eq5} we obtain that the geometric functions
$\gamma_1, \, \gamma_2, \, \nu_1,\, \nu_2, \,
\lambda, \, \mu, \, \beta_1$, $\beta_2$ of the surface satisfy the integrability
conditions given in formulas \eqref{E:eq5-i}. Putting $\beta_1 = \beta_2 =0$ in \eqref{E:eq5-i}, we get
\begin{equation}  \label{E:Eq5-a}
\begin{array}{l}
\vspace{2mm}
2\mu\, \gamma_2 = x(\mu);\\
\vspace{2mm}
2\mu\, \gamma_1 = y(\mu);\\
\vspace{2mm}
2\lambda\, \gamma_2 - (\nu_1 + \nu_2)\,\gamma_1 = x(\lambda) - y(\nu_1);\\
\vspace{2mm}
2\lambda\, \gamma_1 - (\nu_1 + \nu_2)\,\gamma_2 = - x(\nu_2) + y(\lambda);\\
\vspace{2mm}
(\nu_1 + \nu_2)\,\mu = 0; \\
\vspace{2mm}
\varepsilon (\lambda^2 -\mu^2 -\nu_1 \,\nu_2) = x(\gamma_2) - y(\gamma_1) + (\gamma_1)^2 - (\gamma_2)^2.
\end{array}
\end{equation}

The first and second equalities of  \eqref{E:Eq5-a} imply $\gamma_1 =
\ds{\frac{1}{2} y(\ln |\mu|)}, \,\, \gamma_2 = \ds{\frac{1}{2} x(\ln |\mu|)}$.
On the other hand,  $\gamma_1 = - y(\ln \sqrt{E}), \,\, \gamma_2 = - x(\ln \sqrt{-G})$.
Hence, $x(\ln|\mu|(-G)) = 0$ and $y(\ln|\mu| E) = 0$, which imply that
$E |\mu|$ does not depend on $v$, and $G |\mu|$ does not depend on $u$.
Hence,
there exist functions $\varphi(u) >0 $ and $ \psi(v) >0$, such that
$$E |\mu| = \varphi(u); \qquad -G|\mu| = \psi(v).$$
Under the following change of the parameters:
\begin{equation*}  \label{E:Eq6}
\begin{array}{l}
\vspace{2mm}
\overline{u} = \ds{\int_{u_0}^u \sqrt{\varphi(u)}\, du}
+ \overline{u}_0, \quad \overline{u}_0 = const\\
[2mm]
\overline{v} = \ds{\int_{v_0}^v  \sqrt{\psi(v)}\, dv + \overline{v}_0},
\quad \overline{v}_0 = const
\end{array}
\end{equation*}
we obtain
$$\overline{E} = \ds{\frac{1}{|\mu|}}; \qquad \overline{F} = 0;
\qquad \overline{G} = \ds{-\frac{1}{|\mu|}},$$
which imply that the  parameters $(\overline{u},\overline{v})$ are canonical.

\end{proof}

It is clear that the canonical parameters are determined locally  up to an orientation.

\begin{prop}\label{P:prop}
Each Lorentz surface of general type with  parallel normalized mean
curvature vector field in $\E^4_2$ is a surface with flat normal connection.
\end{prop}

\begin{proof} If $M^2_1$ is a Lorentz surface of general type with  parallel normalized mean
curvature vector field, then formulas \eqref{E:Eq5-a} hold. Since $\mu \neq 0$, from the fifth
equality of \eqref{E:Eq5-a} we get that $\nu_1 + \nu_2 = 0$. Hence, the normal connection of
the surface is $\varkappa = 0$, i.e. $M^2_1$ is a surface with flat normal connection.

\end{proof}

Now let $M^2_1: z = z(u,v), \,\, (u,v) \in {\mathcal D}$ be a Lorentz  surface of general type
with  parallel normalized mean
curvature vector field  and $(u,v)$ be canonical parameters. Since $\nu_1 + \nu_2 = 0$, we denote
$\nu := \nu_1 = - \nu_2$.
The functions $\gamma_1$ and $\gamma_2$
are expressed as follows:
$$\gamma_1 = \frac{1}{2} y(\ln |\mu|) = \left(\sqrt{|\mu|}\right)_v; \qquad \gamma_2
= \frac{1}{2} x(\ln |\mu|) = \left(\sqrt{|\mu|}\right)_u.$$
The third and fourth equalities of \eqref{E:Eq5-a} imply the following partial differential
equations for the functions $\lambda$, $\mu$, and $\nu$:
\begin{equation}  \notag
\begin{array}{l}
\vspace{2mm}
\nu_u = -\lambda_v + \lambda (\ln|\mu|)_v;\\
\vspace{2mm}
\nu_v = \lambda_u - \lambda (\ln|\mu|)_u.\\
\end{array}
\end{equation}
The last  equality of  \eqref{E:Eq5-a}  implies
\begin{equation}  \notag
\varepsilon (\lambda^2 - \mu^2 + \nu^2) = \frac{1}{2}|\mu| \Delta^h \ln |\mu|,
\end{equation}
where $\Delta^h = \ds{\frac{\partial^2}{\partial u^2} - \frac{\partial^2}{\partial v^2}}$
is the hyperbolic Laplace operator.

\vskip 3mm
The fundamental existence and uniqueness theorem for the class of
Lorentz surfaces of general type  with  parallel normalized mean curvature vector field
can be formulated in terms of canonical parameters as follows.

\begin{thm}\label{T:Fundamental Theorem-canonical}
Let $\lambda(u,v)$, $\mu(u,v)$ and $\nu(u,v)$ be  smooth functions, defined in a domain
${\mathcal D}, \,\, {\mathcal D} \subset {\R}^2$, and satisfying the conditions
$$\begin{array}{l}
\vspace{2mm}
\mu  \neq 0, \quad \nu \neq const;\\
\vspace{2mm}
\nu_u = -\lambda_v +\lambda (\ln|\mu|)_v;\\
\vspace{2mm}
\nu_v = \lambda_u - \lambda (\ln|\mu|)_u;\\
\vspace{2mm}
\varepsilon (\lambda^2 - \mu^2 +\nu^2) = \ds{\frac{1}{2}|\mu| \Delta^h \ln |\mu|}.
\end{array} $$
If $\{x_{0},\,y_{0},\,b_{0},\,l_{0}\}$ is an orthonormal frame at
a point $p_{0}\in\E^4_2$ (with $\langle b_0, b_0 \rangle = \varepsilon$;
$\langle l_0, l_0 \rangle = -\varepsilon$; $\langle b_0, l_0 \rangle = 0$), then there exist a
subdomain ${\mathcal D}_0 \subset {\mathcal D}$
and a unique Lorentz surface
$M^2_1: z = z(u,v), \,\, (u,v) \in {\mathcal D}_0$ of general type with parallel normalized mean
curvature vector field, such that $M^2_1$ passes through $p_0$,  the functions   $\lambda(u,v)$,
$\mu(u,v)$, $\nu(u,v)$ are the geometric functions
of $M^2$ and $\{x_{0},\,y_{0},\,b_{0},\,l_{0}\}$ is the geometric
frame of $M^2_1$ at the point $p_0$.
Furthermore, $(u,v)$ are canonical parameters of $M^2$.
\end{thm}

\vskip 3mm
The meaning of Theorem \ref{T:Fundamental Theorem-canonical} is that
any  Lorentz surface of general type with  parallel normalized mean curvature vector field
is determined up to a rigid motion in $\E^4_2$ by three invariant functions $\lambda$, $\mu$, $\nu$
satisfying the following  system of three natural partial differential equations:
\begin{equation*}
\begin{array}{l}
\vspace{2mm}
\nu_u = -\lambda_v + \lambda (\ln|\mu|)_v;\\
\vspace{2mm}
\nu_v = \lambda_u - \lambda (\ln|\mu|)_u;\\
\vspace{2mm}
\varepsilon (\lambda^2 - \mu^2 +\nu^2) = \ds{\frac{1}{2}|\mu| \Delta^h \ln |\mu|}.
\end{array}
\end{equation*}

\vskip 3mm
\textbf{Remarks}:
\begin{itemize}
\item
The introducing of canonical parameters on a Lorentz surface with  parallel normalized mean curvature
vector field allows us to minimize the number of invariants and the number of partial differential
equations which determine the surface. This solves the problem of Lund-Regge for the class of
Lorentz surfaces with  parallel normalized mean curvature vector field.

\item
 With respect to canonical parameters the coefficients of the first fundamental form and the coefficients
 of the second fundamental tensor  are expressed in terms of the invariants of the surface.
\end{itemize}

\vskip 5mm \textbf{Acknowledgments:}
The second and third authors are  partially supported by the National Science Fund,
Ministry of Education and Science of Bulgaria under contract
DFNI-I 02/14.


\vskip 5mm
\begin{thebibliography}{99}


\bibitem{Bur-Mayer}
Burstin, C., Mayer, W., \textit{\"{U}ber affine Geometrie XLI: Die
Geometrie zweifach ausgedehnter Mannigfaltigkeiten $F_2$ im
affinen $R_4$}. Math. Z. \textbf{26}   (1927), 373--407.

\bibitem{Chen1}
Chen B.-Y., {\it Geometry of submanifolds}. Marcel Dekker, Inc.,
New York, 1973.

\bibitem{Chen-MM}
Chen B.-Y., \textit{Surfaces with parallel normalized mean curvature vector}.
Monatsh. Math. \textbf{90}  (1980), no. 3, 185--194.

\bibitem{Chen-JMAA}
Chen, B.-Y., \textit{Classification of marginally trapped Lorentzian flat
surfaces in $\E^4_2$ and its application to biharmonic surfaces}.
J. Math. Anal.  Appl.,  \textbf{340} (2008), no. 2, 861--875.

\bibitem{Chen-HMJ}
Chen, B.-Y., \textit{Classification of marginally trapped surfaces of
constant curvature in Lorentzian complex plane}. Hokkaido
Math. J., \textbf{38}  (2009), no. 2, 361--408.

\bibitem{Chen-TJM}
Chen, B.-Y.,\textit{ Black holes, marginally trapped surfaces and
quasi-minimal surfaces}. Tamkang J. Math., \textbf{40} (2009), no. 4,  313--341.

\bibitem{Chen1-2}
Chen B.-Y., \textit{Classification of spatial surfaces with parallel mean curvature vector in pseudo-Euclidean spaces
with arbitrary codimension}. J. Math. Phys. \textbf{50} (2009), 043503.

\bibitem{Chen1-3}
Chen B.-Y., \textit{Complete classification of spatial surfaces with parallel mean curvature vector in arbitrary non-flat
pseudo-Riemannian space forms}. Cent. Eur. J. Math.\textbf{ 7} (2009), 400–428.


\bibitem{Chen-CEJM}
Chen, B.-Y., \textit{Complete classification of parallel Lorentz surfaces in neutral pseudo hyperbolic 4-space}.
Cent. Eur. J. Math.  \textbf{8}  (2010),  no. 4, 706--734.

\bibitem{Chen-JMP}
Chen, B.-Y., \textit{Complete classification of parallel Lorentz surfaces in four-dimensional neutral pseudosphere}. J. Math. Phys.
\textbf{51}  (2010), no. 8, 083518, 22 pp.

\bibitem{Chen-JGP}
Chen, B.-Y., \textit{Complete explicit classification of parallel Lorentz surfaces in arbitrary pseudo-Euclidean spaces}.
J. Geom. Phys.  \textbf{60}  (2010),  no. 10, 1333--1351.

\bibitem{Chen-KJM}
Chen, B.-Y., \textit{Complete classification of Lorentz surfaces with parallel mean curvature vector in arbitrary pseudo-Euclidean space}.
Kyushu J. Math.  \textbf{64}  (2010),  no. 2, 261--279.

\bibitem{Chen-survey}
Chen, B.-Y., \textit{Submanifolds with parallel mean curvature vector in Riemannian and indefinite space forms}.
Arab J. Math. Sci.  \textbf{16}  (2010),  no. 1, 1--46.

\bibitem{Chen2}
Chen B.-Y., \textit{ Pseudo-Riemannian geometry, $\delta$-invariants and
applications}. World Scientific Publishing Co. Pte. Ltd.,
Hackensack, NJ, 2011.

\bibitem{Chen3}
Chen B.-Y., \textit{Classification of minimal Lorentz surfaces in indefinite space forms with arbitrary codimension and arbitrary index}. Publ. Math. Debrecen,  \textbf{78}  (2011), 485--503.

\bibitem{Chen-Dillen}
Chen, B.-Y., Dillen, F.,\textit{ Classification of marginally trapped
Lagrangian surfaces in Lorentzian complex space forms}. J. Math. Phys., \textbf{48} (2007), no. 1, 013509, 23 pp.; Erratum,
J. Math. Phys.,  \textbf{49} (2008), no. 5, 059901, 1p.

\bibitem{Chen-Dillen-Veken}
Chen, B.-Y., Dillen, F., Van der Veken, J.,
\textit{Complete classification of parallel Lorentzian surfaces in Lorentzian complex space norms}.
Internat. J. Math.  \textbf{21 } (2010),  no. 5, 665--686.

\bibitem{Chen-Garay}
Chen, B.-Y., Garay, O., \textit{Classification of quasi-minimal surfaces
with parallel mean curvature vector in pseudo-Euclidean 4-space
$\E^4_2$}. Result. Math., \textbf{55}  (2009), no. 1-2,  23--38.

\bibitem{Chen-Mihai}
Chen, B.-Y.,  Mihai, I., \textit{Classification of quasi-minimal slant
surfaces in Lorentzian complex space forms}. Acta Math. Hungar., \textbf{122} (2009), no. 4, 307--328.

\bibitem{Chen-Veken}
Chen, B.-Y., Van der Veken, J., \textit{Complete classification of parallel surfaces in 4-dimensional Lorentzian space forms}. Tohoku Math. J.  \textbf{61}  (2009),  no. 1, 1--40.

\bibitem{Chen-Yang}
Chen, B.-Y., Yang, D., \textit{Addendum to ''Classification of marginally
trapped Lorentzian flat surfaces in $\E^4_2$ and its application
to biharmonic surfaces''}. J. Math. Anal.  Appl.,  \textbf{361} (2010), no. 1, 280--282.

\bibitem{Fu-Hou}
Fu, Y., Hou, Z.-H.,  \textit{Classification of Lorentzian surfaces with parallel mean curvature vector in pseudo-Euclidean spaces}.
J. Math. Anal. Appl.  \textbf{371}  (2010),  no. 1, 25--40.

\bibitem{GM3}
Ganchev, G., Milousheva, V., \textit{An invariant theory of spacelike
surfaces in the four-dimensional Minkowski space}.   Mediterr. J. Math. \textbf{9}  (2012), 267--294.

\bibitem{GM4}
Ganchev, G., Milousheva, V., \textit{An invariant  theory of marginally
trapped surfaces in the four-dimensional Minkowski space}. J. Math. Phys.,  \textbf{53} (2012), no. 3, 033705, 15 pp.

\bibitem{GM5}
Ganchev, G., Milousheva, V., \textit{Quasi-minimal rotational surfaces in pseudo-Euclidean
four-dimensional space}. Cent. Eur. J. Math., \textbf{12} (2014), no. 10, 1586--1601.


\bibitem{Haesen-Ort-3}
Haesen, S., Ortega, M.,  \textit{Screw invariant marginally trapped
surfaces in Minkowski 4-space}. J. Math. Anal. Appl.
\textbf{355}  (2009), 639--648.

\bibitem{Lane}
Lane, E., \textit{Projective differential geometry of curves and
surfaces}. University of Chicago Press, Chicago, 1932.

\bibitem{Little}
Little, J., \textit{On singularities of submanifolds of higher
dimensional Euclidean spaces}. Ann. Mat. Pura Appl., IV Ser \textbf{83}  (1969),
261--335.


\bibitem{Rosca}
Rosca, R., \textit{On null hypersurfaces of a Lorentzian manifold}. Tensor (N.S.) \textbf{23}  (1972), 66--74.

\bibitem{Shu}
Shu, S., \textit{Space-like submanifolds with parallel normalized mean curvature vector field in de Sitter space}.
J. Math. Phys. Anal. Geom.  \textbf{7 } (2011),  no. 4, 352--369.

\bibitem{Walter}
Walter, R.,\textit{ \"{U}ber zweidimensionale parabolische Fl\"{a}chen im
vierdimensionalen affinen Raum. I: Allgemeine
Fl\"{a}chentheorie}. J. Reine Angew. Math. \textbf{227}   (1967), 178--208.

\bibitem{Yau}
Yau, S.,
\textit{Submanifolds with constant mean curvature}.
Amer. J. Math.  \textbf{96}  (1974), 346--366.

\end{thebibliography}
\end{document}